\documentclass[10pt]{amsart}
\usepackage{mathrsfs}
\usepackage{amssymb}
\usepackage[all]{xy}
\usepackage{amsthm}
\usepackage{array}
\usepackage{amsmath}
\usepackage{enumitem}

\newtheorem{theorem}{Theorem}[section]

\newtheorem{proposition}[theorem]{Proposition}
\newtheorem{corollary}[theorem]{Corollary}
\newtheorem{remark}[theorem]{Remark}
\newtheorem{definition}[theorem]{Definition}

\newtheorem{question}[theorem]{Question}

\def\be{\begin{equation}}
\def\ee{\end{equation}}
\def\br{\begin{eqnarray}}
\def\er{\end{eqnarray}}

 \title{Holomorphic Gromov's Partial Order}
 \author{Lingxu Meng}

\address{Department of Mathematics\\  North University of China \\ Taiyuan, Shanxi 030051,
P.R. China}
\email{20160012@nuc.edu.cn}
\thanks{}
\date{}

\begin{document}
\maketitle

\begin{abstract}
As in \cite{BBM}, we study holomorphic maps of positive degree between compact complex manifolds, and prove that any holomorphic map of degree one  from a compact complex manifold to itself is biholomorphic. This conclusion confirms  that under  a mild restriction the holomorphic Gromov relation ``$\geq$" is indeed a partial order.
\end{abstract}

\textbf{Keywords}: holomorphic map, degree, biholomorphic, Gromov's partial order.

\textbf{AMSC}: 32Q, 32H.

\section{Introduction}
In 1978, Gromov introduced a notion of \emph{domination} between smooth manifolds in a lecture at the Graduate Center CUNY  as follows:

Let $X$ and $Y$ be $n$-dimensional closed smooth manifolds. We say that \emph{$X\geq Y$} if there is a smooth map of positive degree from $X$ to $Y$.

Gromov asked whether the relation ``$\geq$" is a partial order
in the context of real manifolds of constant negative sectional curvature. Here, as in \cite{BBM}, we consider this problem for general complex manifolds. We introduce the following notions.

\begin{definition}\label{def1}
Let $X$ and $Y$ be connected compact complex manifolds of the same dimension. Then $X\geq_1 Y$ (resp. $X\geq Y$) means that from $X$ to $Y$ there exists a holomorphic map of degree one (resp. positive degree).
\end{definition}

 Then we  rephrase the question of Gromov's partial order as follows:

\begin{question}\label{qu}
Let $X$ and $Y$ be connected compact complex manifolds of  the same  dimension.
\begin{enumerate}[leftmargin=*]
\item[\textup{(a)}] If $X\geq_1 Y$ and $Y\geq_1 X$, are $X$ and $Y$ biholomorphic?
\item[\textup{(b)}] If $X\geq Y$ and $Y\geq X$, are $X$ and $Y$ biholomorphic?
\end{enumerate}
\end{question}

In this note,  we study holomorphic maps of degree one  between connected compact complex manifolds and obtain the following.

\begin{theorem}\label{1hol}
Let $X$ and $Y$ be connected compact complex manifolds of the same  dimension and with the same second Betti number. If $f:X\rightarrow Y$ is a holomorphic map of degree one, then $f$ is biholomorphic.
\end{theorem}

Using this theorem, we can answer Question \ref{qu} partly.

\begin{theorem}\label{the2}
Let $X$ and $Y$ be connected compact complex manifolds of the same dimension.
\begin{enumerate}[leftmargin=*]
\item[\textup{(a)}] If $X\geq_1 Y$ and $Y\geq_1 X$, then $X$ and $Y$ are biholomorphic.
\item[\textup{(b)}]
 Suppose that $X\geq Y$ and $Y\geq X$. If $X$ and $Y$ are not biholomorphic, then $X$ and $Y$ both admit a holomorphic self-map of degree greater than one.
 \end{enumerate}
\end{theorem}

Note that in \cite{BBM}  G. Bharali, I. Biswas and M. Mahan proved the above two theorems under the extra condition that $X$ or $Y$ belongs to the Fujiki class $\mathcal C$.
Recently, Theorem 1.3, which is used to prove Theorem 1.4, has independently been established by G. Bharali, I. Biswas and G. Schumacher in \cite{BBS}. The non-trivial part of the proof in \cite{BBS} rests upon a result in [4] by Bharali-Biswas. In contrast, the proof given here is entirely self-contained, and completely different from that of Bharali-Biswas-Schumacher.

The second part of Theorem \ref{the2} motivates one, from the viewpoint of Question 1.2, to consider those manifolds that do not admit a holomorphic self-map of degree greater than one.  In \cite{BBM}, G. Bharali, I. Biswas and M. Mahan obtained that the projective manifolds have this property if they belong to the following four classes: (i) projective manifolds of general type; (ii) Kobayashi hyperbolic manifolds; (iii) rational homogeneous projective manifolds with Picard number one which are not biholomorphic to $\mathbb{C}P^n$; or (iv) smooth projective hypersurfaces of dimension greater than one and of degree greater than two. We note that the measure hyperbolic manifolds, see section 2.4 in \cite{NO}, which are the generalization of Kobayashi hyperbolic manifolds, also admit this property. For more results on this question, we refer to \cite{AP, AKP, B, Fujimoto, HM}. Here, we consider a special class of compact complex manifolds, i.e, the Calabi-Yau manifolds. A \emph{Calabi-Yau manifold} $X$ in this note means a compact complex manifold $X$ with finite fundamental group and with $K_X^{\otimes m}=\mathcal{O}_X$ for some positive integer $m$. This is a broader definition, in one sense, than the one that the term "Calabi-Yau manifold" had in the early literature: Calabi-Yau manifolds, by our definition, need not be K\"ahlerian, but it is mildly restricted in another sense: that these manifolds must have finite fundamental group.

\begin{theorem}\label{cy}
If $X$ is a Calabi-Yau manifold, then every surjective holomorphic self-map $f:X\rightarrow X$ is biholomorphic.
\end{theorem}

 Combining the above discussions, we immediately get the following result.

\begin{corollary}
Let $X$ and $Y$ be connected compact complex manifolds of the same dimension. Suppose that $X\geq Y$ and $Y\geq X$. If $X$ or $Y$  is a measure hyperbolic manifold or a Calabi-Yau manifold, then $X$ and $Y$ are biholomorphic.
\end{corollary}

The paper is organized as follows.
 Theorem \ref{1hol} and \ref{the2} will be proved in Section 2, and Theorem \ref{cy} will be proved in Section 3.

\section{holomorphic maps of degree one}
In this paper, we need consider degrees of holomorphic maps. First, let us recall the definition of the degree of a smooth map. We often use two equivalent definitions of the degree of a smooth map between oriented smooth manifolds (resp. cf. \cite{BT} and \cite{M}). But we only use the following definition here.
\begin{definition}[\cite{M}]
Let $M$ and $N$ be $n$-dimensional compact oriented smooth manifolds, and let $f:M\rightarrow N$ be a smooth map. If $N$ is connected, then the \emph{degree} of $f$ is defined as
\begin{displaymath}
\deg f:=\sum_{x\in f^{-1}(y)}\textup{sgn}(f)(x),
\end{displaymath}
where $y\in N$ is a regular value and $\textup{sgn} (f)(x)$ is the sign of the Jacobi $J(f)(x)$ of $f$ at $x$.
\end{definition}

The above definition does not depend on the regular value $y$.

\begin{remark}\label{rem0}
If $f:X\rightarrow Y$ is a holomorphic map between $n$-dimensional complex manifolds $X$ and $Y$, then $f$ has a real Jocobi $J_{\mathbb{R}}(f)(x)$, which is a real determinant of order $2n$, and a complex Jacobi $J_{\mathbb{C}}(f)(x)$, which is a complex determinant of order $n$. In the definition of degree of $f$, $f$ is regarded as a smooth map between $2n$-dimensional smooth manifolds $X$ and $Y$, so $\textup{sgn}(f)(x)$ refer to the sign of the real Jacobi $J_{\mathbb{R}}(f)(x)$ in this situation.
\end{remark}

\begin{remark}\label{rem1}
If $f:X\rightarrow Y$ is a holomorphic map between $X$ and $Y$ which are both $n$-dimensional connected compact complex manifolds, then $f$ is surjective if and only if $\deg f\neq 0$. Moreover, in this case, we have $\deg f> 0$. Indeed, suppose that $f$ is surjective. Let $S'$ be the set of its critical values. Then $Y-S'$ is the set of its regular values. For any $y\in Y- S'$ and $x\in f^{-1}(y)$, we have $\textup{sgn} (f)(x)=1$ as the real Jacobi $J_{\mathbb{R}}(f)(x)$ is nonnegative. Therefore, $\deg f>0$.
The other direction is obvious by the definition of  degree.
\end{remark}

 In this section, we consider holomorphic maps of degree one. We first recall the definition of  a modification. A holomorphic map $f:X\rightarrow Y$ between connected compact complex manifolds is called a \emph{modification}, if there is a nowhere dense analytic subset $F\subset Y$, such that $f^{-1}(F)\subset X$ is nowhere dense and $f:X-f^{-1}(F)\rightarrow Y-F$ is biholomorphic. If $F$ is the minimal analytic subset satisfying the above condition, then we call $E=f^{-1}(F)$ the \emph{exceptional set} of the modification $f$. We need the following basic result on the modifications.

\begin{theorem}[\cite{GR}, page 215 and \cite{Fis}, page 170]\label{mod}
If $f:X\rightarrow Y$ is a modification between connected compact complex manifolds and the exceptional set $E$ is not empty, then $E$ has pure codimension one in $X$ and $\textup{codim}_Yf(E)\geq 2$.
\end{theorem}

Recall a proposition of A. Fujiki in \cite{Fuj} about modifications, whose original proof uses the method of local cohomology. For a convenience, we give a simpler proof as in \cite{FMX}.

\begin{proposition}[\cite{Fuj}, Proposition 1.1]\label {thm Fuj}
If $f:X\rightarrow Y$ is a modification of $n$-dimensional compact complex manifolds with exceptional set $E$, then there is an exact sequence
\begin{equation}\label{ses}
\xymatrix{
0\ar[r] &H_{2n-2}(E, \mathbb{R})\ar[r]^{i_*} &H_{2n-2}(X, \mathbb{R})\ar[r]^{f_*} &H_{2n-2}(Y, \mathbb{R})\ar[r]&0
},
\end{equation}
where $i: E\rightarrow X$ is the inclusion. Moreover, if $E_1,\dots,E_r$ are the irreducible components of $E$, then  $H_{2n-2}(E, \mathbb{R})=\oplus^r_{j=1} \mathbb{R}[E_j]$, where $[E_j]\in H_{2n-2}(E, \mathbb{R})$ is the fundamental class of $E_j$ in $E$ for $j=1,\dots, r$.
\end{proposition}

\begin{proof}
By the projection formula
\begin{displaymath}
f_*f^*= \textup{id}:H_r(Y, \mathbb{R})\rightarrow H_r(Y, \mathbb{R}),
\end{displaymath}
$f_*: H_r(X, \mathbb{R})\rightarrow H_r(Y, \mathbb{R})$ is surjective for $r=0,\dots, 2n$. Here $f^*:H_r(Y, \mathbb{R})\rightarrow H_r(X, \mathbb{R})$ is induced by the pull back as follows:
\begin{equation*}
\xymatrix{
H_r(Y, \mathbb{R}) \ar[d]_{PD_Y} \ar[r]^-{f^*}  & H_r(X, \mathbb{R}) \ar[d]^{PD_X}  \\
  H^{2n-r}(Y, \mathbb{R}) \ar[r]_{f^*}  & H^{2n-r}(X, \mathbb{R})},
\end{equation*}
where $PD_X$ (resp. $PD_Y$) is the Poincar\'e duality of $X$ (resp. $Y$).

Let $F= f(E)$, $U= X-E$, and $V= Y- F$. Then $f\mid_U: U\rightarrow V$ is a biholomorphic map. Consider the following commutative diagram of exact sequences of Borel-Moore homology
\begin{displaymath}
\xymatrix{
H_{2n-1}(X, \mathbb{R})\ar[d]^{f_*} \ar[r]& H^{BM}_{2n-1}(U, \mathbb{R}) \ar[d]^{\cong} \ar[r]& H_{2n-2}(E, \mathbb{R}) \ar[d]\ar[r]^{i_*}& H_{2n-2}(X, \mathbb{R})\ar[d]^{f_*}\ar[r]&H^{BM}_{2n-2}(U, \mathbb{R})\ar[d]^{\cong}\\
H_{2n-1}(Y, \mathbb{R})       \ar[r]& H^{BM}_{2n-1}(V, \mathbb{R})     \ar[r] & H_{2n-2}(F, \mathbb{R})     \ar[r]& H_{2n-2}(Y, \mathbb{R})         \ar[r]&H^{BM}_{2n-2}(V, \mathbb{R}) }.
\end{displaymath}
By Theorem \ref{mod}, codim$_Y F\geq 2$, so $H_{2n-2}(F, \mathbb{R})= 0$. By the second long exact sequence, we know that $H_{2n-1}(Y, \mathbb{R})\rightarrow H^{BM}_{2n-1}(V, \mathbb{R})$ is surjective. Since $f_*: H_{2n-1}(X, \mathbb{R})\rightarrow H_{2n-1}(Y, \mathbb{R})$ is surjective, $H_{2n-1}(X, \mathbb{R})\rightarrow H^{BM}_{2n-1}(U, \mathbb{R})$ is surjective. Hence $i_*$ is injective by the first long exact sequence.

If $\alpha\in H_{2n-2}(X, \mathbb{R})$ and $f_*(\alpha)=0$, then the image of $\alpha$ in $H^{BM}_{2n-2}(U, \mathbb{R})\cong H^{BM}_{2n-2}(V, \mathbb{R})$ is zero. Hence, $\alpha$ is in the image of $i_*$ by the first long exact sequence. Thus, $\textup{Ker} f_*\subseteq \textup {Im} i_*$. From the fact $H_{2n-2}(F, \mathbb{R})= 0$, we also have $f_*i_*=0$, i.e., $\textup{Im} i_*\subseteq \textup{Ker} f_*$. Therefore  $\textup{Im} i_*=\textup{Ker} f_*$. Combining the above discussions, We get the short exact sequence (\ref{ses}).

By Theorem \ref{mod}, $E_1,\dots,E_r$ all have dimension $(n-1)$. Set
\begin{displaymath}
A:=\cup_{i\neq j}(E_i\cap E_j),
\end{displaymath}
\begin{displaymath}
E'_i:=E_i-A\cap E_i.
\end{displaymath}
Then, all $E'_i$ for $i=1,\cdots,r$ do not  intersect with one another and $E-A=\cup_iE'_i$.
We consider the exact sequence of Borel-Moore homology for $(E, A)$
\begin{displaymath}
\xymatrix{
H_{2n-2}(A, \mathbb{R})\ar[r] &H_{2n-2}(E, \mathbb{R})\ar[r] &H^{BM}_{2n-2}(E-A, \mathbb{R})\ar[r] &H_{2n-3}(A, \mathbb{R}).
}
\end{displaymath}
By the definition of $A$, $\dim A\leq n-2$, so $H_{2n-2}(A, \mathbb{R})=H_{2n-3}(A, \mathbb{R})=0$. Hence
\begin{displaymath}
H_{2n-2}(E, \mathbb{R})= H^{BM}_{2n-2}(E-A, \mathbb{R})= \oplus_i H^{BM}_{2n-2}(E'_i, \mathbb{R}).
\end{displaymath}
Considering the long exact sequence of Borel-Moore homology for $(E_i, A\cap E_i)$, we obtain
\begin{displaymath}
H^{BM}_{2n-2}(E'_i, \mathbb{R})= H_{2n-2}(E_i, \mathbb{R})= \mathbb{R}[E_i].
\end{displaymath}
Hence, $H_{2n-2}(E, \mathbb{R})= \oplus_i \mathbb{R}[E_i]$.
\end{proof}

Having made the above preparations,  we can prove Theorem \ref{1hol}.
\begin{proof}
Denote $J_{\mathbb{C}}(f)(x)$ the complex Jacobi of $f$ at $x$. Define
\begin{displaymath}
 S:=\{x\in X\mid J_{\mathbb{C}}(f)(x)=0\},
\end{displaymath}
and $S'=f(S)$. That is,  $S$ is the set of critical points, and $S'$ is the set of critical values. Then  $Y-S'$ is the set of regular values.

By Remark \ref{rem1}, $f:X\to Y$ is surjective and for any $y\in Y-S'$ and $x\in f^{-1}(y)$,  $\textup{sgn}(f)(x)=1$. Since
\begin{displaymath}
\sum_{x\in f^{-1}(y)}\textup{sgn}(f)(x)= \deg f = 1,
\end{displaymath}
$f^{-1}(y)$ contains only one point. Hence, $f:X-f^{-1}(S')\rightarrow Y-S'$ is injective. Therefore it is biholomorphic.

By \cite{U}, Corollary 1.7, $S$ is a nowhere dense analytic subset of $X$. Then by the proper mapping theorem, $S'=f(S)$ is an analytic subset. Since
$$\dim S' =\dim f(S)\leq\dim S<\dim X=\dim Y,$$
 $S'$ is also nowhere dense in $Y$. We claim that the analytic subset $f^{-1}(S')$ is also nowhere dense in $X$. Indeed, if $f^{-1}(S')$ is dense in $X$, then $f^{-1}(S')=X$ since $X$ is connected. Hence, $S'=f(f^{-1}(S'))=f(X)=Y$, which contradicts that $S'$ is nowhere dense in $Y$. Therefore, $f$ is a modification.

Suppose the exceptional set $E\subseteq X$ of $f$ is not empty. By Theorem \ref{mod}, $E$ has pure codimension one, $\textup{codim}_Y f(E)\geq 2$ and $f:X-E\rightarrow Y-f(E)$ is biholomorphic. Assume that $r$ is the number of irreducible components of $E$. By Proposition \ref{thm Fuj}, $b_2(X)=b_2(Y)+r$. Then by the hypothesis of theorem,   $r=0$, i.e., $E= \emptyset$, which contradicts to the previous assumption. Therefore, $f$ is a biholomorphic map.
\end{proof}

\begin{corollary}\label {cor self}
\begin{enumerate}[leftmargin=*]
\item[\textup{(a)}] Any holomorphic self-map of degree one of a connected compact complex manifold must be biholomorphic.
\item [\textup{(b)}] Let $f:X\rightarrow Y$ be a holomorphic map of degree one. If $X$ and $Y$ are both $K3$ surfaces, Enriques surfaces or complex tori of the same dimension, then $f$ is biholomorphic.
    \end{enumerate}
\end{corollary}
\begin{proof}
We get part (a) and (b) immediately by Theorem \ref{1hol} if we note that in all cases $b_2(X)=b_2(Y)$.
\end{proof}

\begin{remark}
We can also obtain Corollary \ref{cor self}, (b) by Proposition \ref{un-cov}, (b) immediately.
\end{remark}

Now, we prove Theorem \ref{the2} as follows.

\begin{proof}
(a) If $f:X\rightarrow Y$ and $g:Y\rightarrow X$ are both holomorphic maps of degree one, then the composition $g\circ f:X\rightarrow X$ is a holomorphic self-map of degree one. By Corollary \ref{cor self}, (a), $g\circ f$ is biholomorphic. Hence, $f$ is injective. Since $f$ is also surjective, it is bijective.  Therefore, $f$ is a biholomorphic map.

(b) By Definition \ref{def1}, there exist positive degree holomorphic maps $f:X\to Y$ and $g:Y\to X$. Hence the push out $f_\ast:H_2(X, \mathbb{R})\rightarrow H_2(Y, \mathbb{R})$  and $g_\ast:H_2(Y,\mathbb R)\to H_2(X,\mathbb R)$ are surjective. Thus, $b_2(X)= b_2(Y)$. Now we have the conclusions that $\deg f>1$ and $\deg g>1$. Otherwise by Theorem \ref{1hol}, $f$ or $g$ is biholomorphic, which contradicts the hypothesis of theorem. Hence, $\deg(f\circ g)>1$ and $\deg(g\circ f)>1$.
\end{proof}



Next we give another sufficient condition for a degree-one holomorphic map to be  biholomorphic.  We first recall some notations.  For a compact complex manifold $X$, define its \emph{Neron-Severi group}
$$NS(X)=\textup{Im} (c_1:\textup{Pic}(X)\rightarrow H^2(X, \mathbb{Z})).$$ Denote $NS(X)_{\mathbb{R}}:=NS(X)\otimes_{\mathbb{Z}}{\mathbb{R}}$ and $\rho(X):=\dim_{\mathbb{R}}NS(X)_{\mathbb{R}}$. The number $\rho(X)$ is called the {\sl Picard number} of $X$.

\begin{proposition}\label {NSFuj}
Let $X$ be a connected projective manifold and $Y$ a connected compact complex manifold. If $f:X\rightarrow Y$ is a modification and $E_1,\dots,E_r$ are the irreducible components of its exceptional set $E$, then there is an exact sequence
\begin{equation}\label{ex2}
\xymatrix{
0\ar[r] &\oplus^r_{j=1} \mathbb{R}[E_j]\ar[r]^{i_\ast} &NS(X)_{\mathbb{R}}\ar[r]^{f_*} &NS(Y)_{\mathbb{R}}\ar[r]&0
},
\end{equation}
where $i: E\rightarrow X$ is the inclusion and $[E_j]$ is the fundamental class of $E_j$ in $X$ for $j=1,\dots, r$.
\end{proposition}

\begin{proof}
For any $L\in \textup{Pic}(X)$, there is a divisor $D$ on $X$ such that $L=\mathcal{O}(D)$ and hence, $f_*(c_1(L))=c_1(\mathcal{O}(f_*(D)))$. Therefore, $f_*(NS(X)_{\mathbb{R}})\subseteq NS(Y)_{\mathbb{R}}$. On the other hand, since $f^*(NS(Y)_{\mathbb{R}})$ $\subseteq NS(X)_{\mathbb{R}}$, by the projection formula $f_*f^*=\textup{id}$, we have
$$NS(Y)_{\mathbb{R}}=f_*f^*(NS(Y)_{\mathbb{R}})\subseteq f_*(NS(X)_{\mathbb{R}}).$$
 Hence, $f_*$ is surjective.  Then since $[E_i]\in NS(X)_{\mathbb{R}}$, sequence (\ref{ex2}) is exact subsequently by Proposition \ref{thm Fuj}.
\end{proof}
\begin{theorem}
Let $X$ be a connected projective manifold and $Y$ a connected compact complex manifold with $\dim X=\dim Y$. Suppose that the Picard numbers $\rho(X)=\rho(Y)$. If $f:X\rightarrow Y$ is a holomorphic map of degree one, then $f$ is a biholomorphic map.
\end{theorem}

\begin{proof}
As the proof of Theorem \ref{1hol}, we know that $f$ is a modification. Then by Proposition \ref{NSFuj}, the number of irreducible components of its exceptional set is zero. Hence, $f$ is a biholomorphic map.
\end{proof}

We will give an application  of Theorem \ref{1hol}.
G. Bharali and I. Biswas considered the rigidity of a holomorphic self-map of a fiber space in \cite{BB}, where Theorem 1.2 is the following  under the extra condition that $\dim H^1(X_s, \mathcal{O}_{X_s})$ is independent on $s\in S$.

\begin{theorem}[\cite{BB}]
 Suppose that $S$ is a connected compact complex manifold and $p:X\rightarrow S$ is a family of connected compact complex manifolds (i.e., $p$ is a proper holomorphic submersion with connected compact fibers). Let $F:X\rightarrow X$ be a holomorphic map such that there exist two points $a, b\in S$ satisfying $F(X_a)\subseteq X_b$. Then
\begin{enumerate}[leftmargin=*]
\item[\textup{(a)}] $F$ is a holomorphic map of fiber spaces, i.e., there exists a holomorphic map $f:S\rightarrow S$ such that $p\circ F=f\circ p$; and
\item[\textup{(b)}] If $F\mid_{X_a}:X_a\rightarrow X_b$ has degree one, then $F$ is a fiberwise biholomorphism.
\end{enumerate}
\end{theorem}

\begin{proof}
Part (a) is same as Theorem 1.2, a) in \cite{BB}.  From the original proof of Theorem 1.2, b) in \cite{BB}, we know that  if $F\mid_{X_a}:X_a\rightarrow X_b$ has degree one, then for any $s\in S$, $F\mid_{X_s}:X_s\rightarrow X_{f(s)}$ has also degree one. Since all fibers of $p:X\rightarrow S$ are diffeomorphic, $b_2(X_s)= b_2(X_{f(s)})$. Hence, by Theorem \ref{1hol}, we get part (b).
\end{proof}

\section{holomorphic maps with positive degree}
In this section we consider surjective holomorphic maps. A holomorphic map $f:X\rightarrow Y$ of complex manifolds is called \emph{finite}, if $f$ is proper and for any point $y\in Y$, $f^{-1}(y)$ is a finite set.

\begin{proposition}\label {thm Ksf}
Let $X$ be a compact K\"ahler  manifold and $Y$ a compact complex manifold. Suppose that the Betti numbers of $X$ and $Y$ are equal. If $f:X\rightarrow Y$ is a surjective holomorphic map, then $f$ is finite. Therefore, $\dim X=\dim Y$ and $Y$ is a K\"ahler manifold.
\end{proposition}

\begin{proof}
 If $f:X\rightarrow Y$ is finite, then $\dim X=\dim Y$ and $Y$ is a K\"ahler manifold by Theorem 2 in \cite{Var}. So, we only need to prove that $f$ is finite. Assume that $f$ is not finite. Then there exists a point $y\in Y$ such that $\dim f^{-1}(y)\geq 1$. We choose an irreducible analytic set $Z\subseteq f^{-1}(y)$ such that $\dim Z=r\geq 1$. If we set $\dim X=n$ and $\dim Y=m$, since $f(Z)=\{y\}$, then $f_*([Z])=0$ in $H^{2m-2r}(Y, \mathbb{R})$. Here $[Z]\in H_{2r}(X, \mathbb{R})$ is the fundamental class of $Z$ on $X$, and $f_*:H^{2n-2r}(X, \mathbb{R})\rightarrow H^{2m-2r}(Y, \mathbb{R})$ is the $2r$-th Gysin map.

Since $X$ is a compact K\"ahler manifold, $[Z]\neq 0$ in $H^{2n-2r}(X, \mathbb{R})$. Hence we can choose a class $\gamma\in H^{2r}(X, \mathbb{R})$ such that $[Z]\cup\gamma\neq 0$ in $H^{2n}(X, \mathbb{R})$. By Lemma 7.28 in  \cite{V}, $f^*:H^{2r}(Y, \mathbb{R})\rightarrow H^{2r}(X, \mathbb{R})$ is injective. Since $b_{2r}(X)=b_{2r}(Y)$, $f^*$ is bijective. Hence, there is $\beta\in H^{2r}(Y, \mathbb{R})$ such that $\gamma=f^*(\beta)$. So $[Z]\cup\gamma=[Z]\cup f^*(\beta)=f_*[Z]\cup\beta=0$. It contradicts  the choice of $\gamma$.
\end{proof}

Consequently, we have the following.

\begin{corollary}
Let $X$ be a compact K\"ahler  manifold. If $f:X\rightarrow X$ is a surjective holomorphic map, then $f$ is finite.
\end{corollary}

The following proposition is essentially proved by A. Fujimoto in \cite{Fujimoto}, where $X$ and $Y$ are both projective manifolds with $\dim X=\dim Y$ and  $\rho(X)=\rho(Y)$.

\begin{proposition}\label {thm psf1}
Let $X$ be a projective manifold and $Y$ a compact complex manifold such that  $\rho(X)=\rho(Y)$ or $b_2(X)=b_2(Y)$. If $f:X\rightarrow Y$ is a surjective holomorphic map, then $f$ is finite and  $\dim X=\dim Y$.
\end{proposition}

\begin{proof}
Obviously, we only need to prove that $f$ is finite. Assume that $f$ is not finite. Then there exists a point $y\in Y$ such that $\dim f^{-1}(y)\geq 1$. Since $X$ is a projective manifold, $f^{-1}(y)$ is a projective variety, hence contains a projective curve $C$. Suppose $n=\dim X\geq \dim Y=m$. Since $f(C)=\{y\}$, $f_*[C]=0$ in $H^{2m-2}(Y,\mathbb{R})$, where $f_*:H^{2n-2}(X,\mathbb{R})\rightarrow H^{2m-2}(Y,\mathbb{R})$ is the second Gysin map. The $r$-th Gysin map is defined by pushing out through the Poincar\'e dualities as follows (see \cite{V}, page 178):
\begin{equation*}
\xymatrix{
H^{2n-r}(X, \mathbb{R}) \ar[d]_{PD_X} \ar[r]^-{f_*}  & H^{2m-r}(Y, \mathbb{R}) \ar[d]^{PD_Y}  \\
  H_r(X, \mathbb{R}) \ar[r]_{f_*}  & H_r(Y, \mathbb{R})}.
\end{equation*}

By \cite{V}, Lemma 7.28,  $f^*:H^2(Y, \mathbb{R})\rightarrow H^2(X, \mathbb{R})$ is injective. When $\rho(X)=\rho(Y)$, $f^*:NS(Y)_{\mathbb{R}}\rightarrow NS(X)_{\mathbb{R}}$ is an isomorphism. Let $L$ be an ample line bundle on $X$. Then there is an $\alpha\in NS(Y)_{\mathbb{R}}$ such that $c_1(L)=f^*\alpha$. So $c_1(L)\cup [C]= f^*\alpha\cup [C]=\alpha\cup f_*[C]=0$. It contradicts the ampleness of $L$. When $b_2(X)=b_2(Y)$, we can obtain the contradiction  as in the proof of Proposition \ref{thm Ksf} when we choose $Z=C$ here.
\end{proof}

Now, we consider when a surjective holomorphic map is an unramified covering map. A surjective holomorphic  map is called \emph{an unramified covering map} if it is a finite covering map in the topological sense.

\begin{proposition}\label{un-cov}
\begin{enumerate}[leftmargin=*]
 \item[\textup{(a)}]
 Let $X$ be a compact complex manifold with non-negative Kodaira dimension. If $f:X\rightarrow X$ is a surjective holomorphic map, then $f$ is an unramified covering map.
\item[\textup{(b)}]
Let $X$ and $Y$ be $n$-dimensional  compact complex manifolds with $K_X^{\otimes k}=\mathcal{O}_X$ and $K_Y^{\otimes l}=\mathcal{O}_Y$ for some positive integers $k$ and $l$ respectively. If $f:X\rightarrow Y$ is a surjective holomorphic map, then $f$ is an unramified covering map.
\item[\textup{(c)}]
Let $f:X\rightarrow Y$ be a surjective finite map of complex manifolds (which may be noncompact). If the number $| f^{-1}(y)|$ of the points contained in $f^{-1}(y)$ is independent with $y\in Y$, then $f$ is an unramified covering map.
\end{enumerate}
\end{proposition}

\begin{proof}
(a) is proved in \cite{K}, Theorem 7.6.11, or in \cite{P}.

(b) Let $f_*:T_X\rightarrow f^*T_Y$ be the tangent map. Then
$$\bigwedge^n f_*:\bigwedge^nT_X\rightarrow f^*\bigl(\bigwedge^nT_Y\bigr)$$
 defines a global section $\sigma\in \Gamma(X, K_X\otimes f^*K_Y^{-1})$. If $D(\sigma)$ is the divisor defined by $\sigma$, then $K_X=f^*K_Y\otimes \mathcal{O}(D(\sigma))$. Since $K_X^{\otimes k}=\mathcal{O}_X$ and $K_Y^{\otimes l}=\mathcal{O}_Y$, $\mathcal{O}(D(\sigma^{kl}))=\mathcal{O}_X$. Hence, the divisor $D(\sigma^{kl})=0$, which implies the divisor $D(\sigma)=0$. Therefore $f$ is a holomorphic submersion map. Since $\dim X=\dim Y$,  $f$ is a surjective local isomorphism. By \cite{Ho}, Lemma 2, $f$ is an unramified covering map.

(c) For any $y_0\in Y$, set $f^{-1}(y_0)=\{x_1,\dots, x_d\}$. By \cite{Fis}, Theorem in page 145, $f$ is an open map. So we can choose an open neighbourhood $W_i$ of $x_i$ for  $i=1,\dots, d$ such that $W_i\cap W_j=\emptyset$ for $i\neq j$. Define $H_i:=f(W_i)$ for $i=1,\dots, d$. Let $V:=\bigcap_{i=1}^dH_i$ and $U_i:=W_i\cap f^{-1}(V)$. Then, $f\mid_{U_i}:U_i\rightarrow V$ for $i=1,\dots, d$ is surjective. For any $y\in V$, $(f\mid_{U_i})^{-1}(y)$ is not empty and
\begin{displaymath}
\sum_{i=1}^d | (f\mid_{U_i})^{-1}(y)| \leq | f^{-1}(y)|=d.
\end{displaymath}
So for each $i$, $|(f\mid_{U_i})^{-1}(y)|=1$ and $f^{-1}(y)= \bigcup_{i=1}^d(f\mid_{U_i})^{-1}(y)$, i.e., $f\mid_{U_i}$ is bijective and $f^{-1}(V)= \bigcup_{i=1}^dU_i$. Hence, $f\mid_{U_i}$ is biholomorphic.  We have proved that $f$ is an unramified covering map.
\end{proof}

\begin{proposition}\label{k-isom}
Let $X$ be an $n$-dimensional connected compact complex manifold and  $f:X\rightarrow X$ an unramified holomorphic covering map. Suppose that $X$ satisfies one of the following conditions:
\begin{enumerate}[leftmargin=*]
\item[\textup{(a)}] There exists a Chern number $P(X):=\int_XP(c_1(X),\dots, c_n(X))\neq 0$, where
$$P(T_1,\dots,T_n)=\sum a_{i_1\cdots i_n}T_1^{i_1}\cdots T_n^{i_n}$$
is a polynomial on $T_1,\dots, T_n$ over $\mathbb{Q}$ satisfying $i_j\in \mathbb{N}$ for $j=1,\dots, n$ and $i_1+2i_2+\dots+ni_n=n$. Especially, the Euler characteristic $\chi(\mathcal{O}_X)\neq 0$ of $\mathcal{O}_X$, or the topological Euler characteristic $\chi^{top}(X)\neq 0$ of $X$, or the signature $\sigma(X)\neq 0$ of $X$ when $n$ is even; or
\item[\textup{(b)}] The fundamental group $\pi_1(X)$ has no proper subgroup isomorphic to itself.
\end{enumerate}
\noindent Then $f$ is biholomorphic.
\end{proposition}

\begin{proof}
(a) Since $f$ is an unramified covering map, $T_X=f^*T_X$ and for each $i$, $c_i(X)=f^*c_i(X)$. Hence, we have
\begin{equation*}
\begin{aligned}
P(X)=& \int_XP(f^*c_1(X),\dots,f^*c_n(X))\\
=& \int_Xf^*(P(c_1(X),\dots,c_n(X))) \\
=& \deg f\cdot \int_XP(c_1(X),\dots,c_n(X))\\
=& \deg f\cdot P(X).
\end{aligned}
\end{equation*}
Since $P(X)\neq 0$, $\deg f=1$. Then by Theorem \ref{1hol}, $f$ is biholomorphic.


(b) Let $x_0$ be any point in $X$ and $F=f^{-1}(x_0)$ the fiber of $f$ at $x_0$. Consider the exact sequence
\begin{displaymath}
\xymatrix{
\cdots\ar[r]&\pi_1(F)\ar[r]^{i_*} &\pi_1(X)\ar[r]^{f_*} &\pi_1(X)\ar[r]&\pi_0(F)\ar[r]& \pi_0(X)\ar[r]&\cdots}
\end{displaymath}
where $i_\ast$ is induced by the inclusion $i:F\to X$.  Clearly, $\pi_1(F)=0$. Since $X$ is a connected manifold, it is path-connected, hence $\pi_0(X)=0$. Since $\pi_1(X)$ does not contain any proper subgroup isomorphic to itself, $f_*$ is an isomorphism. Hence, $\pi_0(F)=0$, i.e., $F$ contains only one point. So $f$ is biholomorphic.
\end{proof}

\begin{remark}
 Y. Fujimoto in \cite{Fujimoto} proved Proposition \ref{k-isom}, (a) for $\chi(\mathcal{O}_X)\neq 0$ in the case of projective manifolds.
\end{remark}

Now we can give a proof of Theorem \ref{cy}.

\begin{proof}
By Proposition \ref{un-cov}, (b), $f$ is an unramified map. Then by Proposition \ref{k-isom}, (b), $f$ is a biholomorphic map.
\end{proof}

\noindent{\bf Acknowledgements.}\,
I would like to thank Prof. Jixiang Fu for his encouragement and many helpful discussions. I also would like to thank Prof. Gautam Bharali for telling me the newest results of himself and his co-authors. Finally, I would like to thank the referee's suggestions and careful reading of my manuscript.

\renewcommand{\refname}{\leftline{\textbf{References}}}

\makeatletter\renewcommand\@biblabel[1]{${#1}.$}\makeatother


\begin{thebibliography}{99}


\bibitem{AP}
Andreas, H., Peternell, T.: Non-algebraic compact K\"ahler threefolds admitting endomorphisms. Sci. China Math. \textbf{54}(8), 1635-1664 (2011)

\bibitem{AKP}
Aprodu, M., Kebekus, S., Peternell, T.: Galois coverings and endomorphisms of projective varieties. Math. Z. \textbf{260}(2), 431-449 (2008)


\bibitem{B}
Beauville, A.: Endomorphisms of hypersurfaces and other manifolds. Internat. Math. Res. Notices \textbf{1}, 25-42 (2011)

\bibitem{BB}
Bharali, G., Biswas, I.: Rigidity of holomorphic maps between fiber spaces. Internat. J. Math. \textbf{25}(1), 1450006, 8 pp (2014)

\bibitem{BBM}
Bharali, G., Biswas, I., Mahan, M.: The Fujiki class and positive degree maps. Complex Manifolds \textbf{2}, 11-15 (2015)

\bibitem{BBS}
Bharali, G., Biswas, I., Schumacher, G.: A criterion for a degree-one holomorphic map to be a biholomorphism. Complex Var. Elliptic Equ. \textbf{62}(7), 914-918 (2017)


\bibitem{BT}
Bott, R., Tu, L.: Differential forms in algebraic topology. Graduate Texts in Math. \textbf{82}, Springer-Verlag, New York-Berlin, 1982

\bibitem{Fis}
Fischer, G.: Complex analytic geometry. Lecture Notes in Math. \textbf{538}, Berlin-Heidelberg-New York, Springer, 1976

\bibitem{FMX}
Fu, J., Meng, L., Xia, W.: Complex balanced spaces. Internat. J. Math. \textbf{26}(12), 1550105, 15 pp (2015)

\bibitem{Fuj}
Fujiki, A.: A theorem on bimeromorphic maps of K\"ahler manifolds and its applications. Publ. Res. Inst. Math. Sci. \textbf{17}(2), 735-754 (1981)

\bibitem{Fujimoto}
Fujimoto, Y.: Endomorphisms of smooth projective 3-folds with non-negative Kodaira dimension. Publ. Res. Inst. Math. Sci. \textbf{38}(1), 33-92 (2002)


\bibitem{GR}
Grauert, H., Remmert, R.: Coherent analytic sheaves. Grundlehren der Math. Wiss. \textbf{265}, Springer-Verlag, Berlin, 1984

\bibitem{HM}
Hwang, J.-M., Mok, N.: Holomorphic maps from rational homogeneous spaces of Picard number 1 onto projective manifolds. Invent. Math. \textbf{136}(1), 209-231 (1999)


\bibitem{Ho}
Ho, C.-W.: A note on proper maps. Proc. Amer. Math. Soc. \textbf{51}, 237-241 (1975)


\bibitem{K}
Kobayashi, S.: Hyperbolic Complex Spaces. Grundlehren der Math. Wiss. \textbf{318}, Springer-Verlag, Berlin, 1998

\bibitem{M}
Milnor, J. W.: Topology from the Differentiable Viewpoint. Princeton University Press, Princeton, New Jersey, 1997

\bibitem{NO}
Noguchi, J., Ochiai, T.: Geometric function theory in several complex variables. Translations of Mathematical Monographs \textbf{80}, American Mathematical Society, Providence, RI, 1990


\bibitem{P}
Peters, K.: \"Uber holomorphe und meromorphe Abbildungen gewisser kompakter komplexer Mannigfaltigkeiten. Arch. Math. \textbf{15}, 222-231 (1964)

\bibitem{U}
Ueno, K.: Classification Theory of Algebraic Varieties and Compact Complex Spaces. Lecture Notes in Math. \textbf{439}, Berlin-Heidelberg-New York, Springer, 1975

\bibitem{Var}
Varouchas, J.: Stabilit\'{e} de la classe des vari\'{e}t\'{e}s K\"{a}hl\'{e}riennes par certains morphismes propres. Invent. Math. \textbf{77}(1), 117-127 (1984)

\bibitem{V}
Voisin, C.: Hodge Theory and Complex Algebraic Geometry I. Cambridge Stud. Adv. Math. \textbf{76}, Cambridge University Press, Cambridge, 2003
\end{thebibliography}
\end{document}